\documentclass[12pt,reqno]{amsart}

\usepackage{amssymb}
\usepackage{amsmath}
\usepackage{enumitem}
\usepackage{mathrsfs}
\usepackage[all]{xy}
\usepackage{hyperref}
\usepackage[margin=1.25in]{geometry}

\newtheorem{theorem}{Theorem}[section]
\newtheorem{lemma}[theorem]{Lemma}

\newtheorem{corollary}[theorem]{Corollary}

\theoremstyle{definition}

\newtheorem{remark}[theorem]{Remark}
\newtheorem{example}[theorem]{Example}
\newtheorem{definition}[theorem]{Definition}
\newtheorem{question}[theorem]{Question}
\numberwithin{equation}{section} \numberwithin{figure}{section}

\DeclareMathOperator{\Aut}{Aut}

\DeclareMathOperator{\Spec}{Spec}

\DeclareMathOperator{\im}{Im}

\DeclareMathOperator{\characteristic}{char}

\DeclareMathOperator{\Br}{Br}

\let\H\relax
\DeclareMathOperator{\H}{H}

\DeclareMathOperator{\PGL}{PGL}

\DeclareMathOperator{\SL}{SL}

\renewcommand{\emptyset}{\varnothing}

\newcommand{\To}{\longrightarrow}
\newcommand{\kbar}{\overline{k}}
\newcommand{\Qbar}{\overline{\QQ}}

\newcommand\Gm{\GG_{\mathrm{m}}}
\newcommand{\fm}{\mathfrak{m}}
\newcommand{\fp}{\mathfrak{p}}

\newcommand{\ucA}{{\mathscr A}}
\newcommand{\ucM}{{\mathscr M}}
\newcommand{\ucR}{{\mathscr R}}
\newcommand{\ucU}{{\mathscr U}}
\newcommand{\ucX}{{\mathscr X}}
\newcommand{\ucY}{{\mathscr Y}}
\newcommand{\latt}[1]{{({#1})}}

\renewcommand{\AA}{\mathbb A}

\newcommand{\CC}{\mathbb C}

\newcommand{\FF}{\mathbb F}
\newcommand{\GG}{\mathbb G}

\newcommand{\NN}{\mathbb N}

\newcommand{\PP}{\mathbb P}
\newcommand{\QQ}{\mathbb Q}

\newcommand{\ZZ}{\mathbb Z}
\newcommand{\cA}{\mathcal A}

\newcommand{\cH}{\mathcal H}

\newcommand{\cM}{\mathcal M}

\newcommand{\cO}{\mathcal O}

\newcommand{\cQ}{\mathcal Q}
\newcommand{\cR}{\mathcal R}

\newcommand{\cX}{\mathcal X}

\usepackage{color}
\definecolor{orange}{rgb}{1,0.5,0}


\title{Rationality and arithmetic of the moduli of abelian varieties}
\author{Daniel Loughran and Gregory Sankaran}
\address{Department of Mathematical Sciences \\
University of Bath \\
Bath \\
BA2 7AY \\
UK.}

\begin{document}
	
\begin{abstract}
We study the rationality properties of the moduli space $\cA_g$ of
principally polarised abelian $g$-folds over $\QQ$ and apply the
results to arithmetic questions. In particular we show that any
principally polarised abelian threefold over $\FF_p$ may be lifted to
an abelian variety over $\QQ$. This is a phenomenon of low dimension:
assuming the Bombieri-Lang conjecture we also show that this is not
the case for abelian varieties of dimension at least seven. About
moduli spaces, we show that $\cA_g$ is unirational over $\QQ$ for
$g\le 5$ and stably rational for $g=3$. This also allows us to make
unconditional one of the results of Masser
and Zannier about the existence of
abelian varieties over $\QQ$ that are not isogenous to Jacobians.
\end{abstract}
	
\maketitle
\tableofcontents
	
\section{Introduction}\label{sec:intro}
Arithmetic properties of abelian varieties are strongly linked to the
geometry of their moduli spaces. Here we study both the birational
geometry over $\QQ$ of the coarse moduli space $\cA_g$ of principally
polarised abelian varieties (ppavs), and lifting and other arithmetic
properties of ppavs themselves.

\subsection{Lifting abelian varieties}
For any prime $p$ and any elliptic curve $E_p$ over $\FF_p$, there
exists an elliptic curve $E$ over $\QQ$ such that $E_p$ is the
reduction modulo $p$ of $E$ (we say that $E$ is a \emph{lift} of
$E_p$ to $\QQ$). Indeed, one simply takes a suitable lift of the
coefficients of $E_p$.

For higher dimensional abelian varieties the problem becomes more interesting.
Firstly we at least need that the abelian variety $A_p$ lifts from $\FF_p$ to $\QQ_p$;
to guarantee this we assume that $A_p$ is equipped with a principal polarisation \cite[Cor.~2.4.2]{Oor}. 
It is not too difficult to see that principally polarised
abelian surfaces always lift to $\QQ$ as the generic such surface is
the Jacobian of a hyperelliptic curve (see \S\ref{subsec:WA_A_2}). One
of our first results is that lifting can be achieved in dimension
three, and moreover for finitely many primes simultaneously.

For an abelian variety $A$ over $\QQ$ and a
prime $p$, we denote by $A_{\FF_p}$ the reduction modulo $p$ of the
N\'{e}ron model of $A$.

\begin{theorem}\label{thm:prescribedreductions}
Let $S$ be a finite set of rational primes, and for each $p \in S$ fix
$A_p$, a principally polarised abelian $3$-fold over $\FF_p$. Then
there exists a principally polarised abelian $3$-fold $A$ over $\QQ$
such that $A_p \cong A_{\FF_p} \text{ for all }p \in S$ as principally
polarised abelian varieties.
\end{theorem}

This property should not hold in higher dimension.

\begin{theorem} \label{thm:BL}
Let $g \geq 7$ and assume the Bombieri--Lang conjecture. Then for all
but finitely many primes $p$, there exists a principally polarised
abelian $g$-fold $A_p$ over $\FF_p$ such that $A_p \not\cong
A_{\FF_p}$ for any principally polarised abelian $g$-folds $A$ over
$\QQ$.
\end{theorem}

\subsection{Rationality properties of moduli spaces}
We achieve our arithmetic results through a consideration of the
birational geometry of the coarse moduli space $\cA_g$ of principally
polarised abelian $g$-folds over $\QQ$. This has been much studied
over the complex numbers. It is known that $\cA_g$ unirational over
$\CC$ for $g\le 5$, and rational over $\CC$ for $g\le{3}$. On the
other hand, if $g\ge 7$ then $\cA_g$ is of general type~\cite{Tai},
and $\cA_6$ has non-negative Kodaira dimension~\cite{DSMS}. Questions
about the rationality of $\cA_g$ over non-closed fields seem by
contrast to have had little attention, and we address some of them
here, over $\QQ$.

We first consider unirationality and, incorporating previous
results, obtain the following theorem.

\begin{theorem} \label{thm:unirational}
The coarse moduli space $\cA_g$ of principally polarised abelian
varieties is unirational over $\QQ$ if and only if $g\le 5$.
\end{theorem}

This result has motivation from the recent
paper~\cite{MZ}, in which Masser and Zannier prove various results
on the existence of abelian varieties over $\Qbar$ that are not
isogenous to Jacobians. Some of their  results also hold for
abelian varieties over $\QQ$, but for $g=4$ and $g=5$ that refinement
is conditional on the unirationality of $\cA_g$ over $\QQ$.

From this and \cite[Thm~1.5, Cor~1.6]{MZ} we obtain the following
immediate application, which is new for $g=4,\,5$ (see
\cite[Thm~1.5]{MZ} for a stronger statement).  

\begin{corollary}\label{cor:nonjacobian}
Let $k$ be a number field. For $g=\,4,\,5$, there exists a
principally polarised abelian $g$-fold over $k$ that is Hodge generic
and not isogenous to any Jacobian.
\end{corollary}

We prove Theorem~\ref{thm:unirational} by showing that the Prym moduli
space $\cR_{g+1}$ is unirational for $g=4,\,5$. This is sufficient to
get unirationality of $\cA_g$, as the Prym map $\cR_{g+1} \to \cA_g$
is dominant for $g \leq 5$ over any field of characteristic not equal
to $2$ \cite[Thm.~6.5]{Bea}.

For deeper rationality properties, it is known that $\cA_2$ is
rational over any field, by work of Igusa \cite{Ig}.  The rationality
of $\cA_3$ over $\CC$ was first proven by Katsylo \cite{Kat} (see also
B\"{o}hning's exposition \cite{Bo}). It seems possible that this
result could also hold over $\QQ$; however Katsylo's proof is
notoriously delicate and technical. We content ourselves with the
following weaker statement, which has a much simpler proof.

\begin{theorem} \label{thm:A3stablyrational}
The coarse moduli space $\cA_3$ is stably rational over $\QQ$.
\end{theorem}

\subsection{Weak approximation for algebraic stacks}
Theorem \ref{thm:A3stablyrational} is the key geometric input for the
proof of Theorem \ref{thm:prescribedreductions}. Namely, one can
interpret Theorem \ref{thm:prescribedreductions} as a version of
\emph{weak approximation} for the moduli stack $\ucA_3$ of principally
polarised abelian threefolds. Recall that a smooth variety $X$ over a
number field $k$ is said to satisfy \emph{weak approximation} if
$X(k)$ is dense in $\prod_v X(k_v)$ where the product is over all
places $v$ of $k$. We prove a version of this for the stack $\ucA_3$
(see \S\ref{sec:WA} for definitions for stacks).

\begin{theorem} \label{thm:A3weakapprox}
The stack $\ucA_3$ satisfies weak approximation over any number field $k$.
\end{theorem}

The classical weak approximation theorem implies that rational
varieties satisfy weak approximation, and a fibration argument extends
this to stably rational varieties. However Theorem
\ref{thm:A3stablyrational} does not immediately imply Theorem
\ref{thm:A3weakapprox}; indeed there are algebraic stacks that fail
weak approximation but whose coarse moduli space satisfies weak
approximation (see Example~\ref{ex:GW}). To exclude this possibility
for $\ucA_3$ one needs to be careful with twists.  Theorem
\ref{thm:prescribedreductions} is then an application of this weak
approximation result.

\subsection{Questions}
Our results raise the following, to which we do not know the answer.

\begin{question}\label{qu:weakweak}
Do $\cA_4$ and $\cA_5$ satisfy weak weak approximation (i.e.\ weak
approximation away from a finite set of places)?
\end{question}

Standard conjectures in arithmetic geometry, together with the
unirationality from Theorem~\ref{thm:unirational}, would imply a
positive answer to Question~\ref{qu:weakweak}. This in turn would give
a version of Theorem~\ref{thm:A3weakapprox} for $\ucA_5$, and a
version of Theorem~\ref{thm:prescribedreductions} for $g=5$ away from
finitely many primes. The case of $\ucA_4$ is less clear as here the
generic gerbe is non-neutral (see Lemma~\ref{lem:neutralgerbe}).

\begin{question}\label{qu:mildBL}
Can one prove an unconditional version of Theorem~\ref{thm:BL}?
\end{question}

\subsection*{Conventions}
We denote by $\ucA_g$ the moduli stack of principally polarised
abelian $g$-folds over $\ZZ$ and by $\cA_g$ its coarse moduli
space. The stack $\ucA_g$ is smooth over $\ZZ$ \cite[Thm.~2.4.1]{Oor}.
We sometimes abuse notation and also denote by $\ucA_g$ the base
change of the stack to some field, which will be clear from the
context.%  We use the shorthand ppav for ``principally polarised
%abelian variety''.

For a group scheme $G$ we denote by $BG$ the associated classifying
stack. Regarding gerbes, we use the conventions of
\cite[Ch.~12]{Ols}. If $G$ is an abelian group scheme, we say that a
$G$-gerbe is neutral if it has a section; this is equivalent to being
isomorphic to $BG$.

For an algebraic stack $\ucX$ and a scheme $S$, we abuse notation and
denote by $\ucX(S)$ the \textit{set} of isomorphism classes of
$S$-points of $\ucX$ (rather than the groupoid).

\subsection*{Acknowledgements}

We are grateful to Christian B\"{o}hning, Davide Lombardo and Siddharth Mathur for useful
discussions. Christian largely taught us how to prove
Theorem~\ref{thm:A3stablyrational}. We are also grateful to David
Masser and Umberto Zannier, who asked about the unirationality of $\cA_4$. Daniel Loughran
was supported by UKRI Future Leaders Fellowship
\texttt{MR/V021362/1}.

\section{Background on moduli}\label{sec:background}

In this preliminary section we outline the known results over the
complex numbers and make some remarks about certain moduli spaces
associated with curves, which we shall use as auxiliaries.

Let $k$ be a field. Recall that a variety $X$ over $k$ is called
\emph{rational} (resp.~\emph{unirational}) if it admits a birational
(resp.~dominant rational) map from a projective space. It is called
\emph{stably rational} if $X \times \PP^n$ is rational for some $n$.

\subsection{Known rationality results}\label{subsec:algclosed}

The coarse moduli spaces $\cA_g$ have been much studied over the
complex numbers from a birational point of view. With the exception of
the case $g=6$, the broad picture remains as described in~\cite{HS},
to which we refer for more details.

\subsubsection{$g=2$}
Igusa~\cite[Thm.~5]{Ig} showed that $\cA_2 = \cM_2$ is rational over
any field, and there are numerous results on the moduli space for
abelian surfaces with non-principal polarisations or level structures.

\subsubsection{$g=3$}
It is easy to see that $\cA_3$ is unirational over
$\CC$. Katsylo~\cite{Kat} showed that $\cM_3$ (and hence $\cA_3$, by
Torelli) is rational over $\CC$: we discuss this
in~\S\ref{subsec:curves} below.

\subsubsection{$g=4$}
The first proof that $\cA_4$ is unirational over $\CC$ was given by
Clemens~\cite{Cle}, using intermediate Jacobians. Other proofs were
subsequently given by Verra~\cite{Ver1} and by Izadi, Lo Giudice and
the second author~\cite{ILS}. Of these, the proof in~\cite{ILS}, which
uses the moduli space of Prym curves (see \S\ref{subsec:prym} below),
seems the easiest to adapt to non-closed fields.

\subsubsection{$g=5$}
The first proofs that $\cA_5$ is unirational over $\CC$ were given by
Donagi~\cite{Don} and by Mori and Mukai~\cite{MM}. There is a
different proof in \cite{Ver2} and a more recent one by Farkas and
Verra~\cite{FV16}. All of these use Prym curves. We found the argument
in \cite{FV16} the easiest to adapt for our purpose.

\subsubsection{$g\ge 6$}
The spaces $\cA_g$ are non-rational for $g>5$, at least in
characteristic zero. Tai~\cite{Tai}, building on earlier work of
Mumford and of Freitag, showed that $\cA_g$ is of general type for
$g\ge 7$. The case of $\cA_6$ remained completely mysterious until
2020, when Dittman, Salvati Manni, and Scheithauer~\cite{DSMS} showed
that the second plurigenus is positive: thus the Kodaira dimension is
non-negative.

\subsection{Moduli of curves}\label{subsec:curves}
In view of Katsylo's result, it is natural to ask whether $\cA_3$ is
rational over $\QQ$. What Katsylo proves directly, however, is that
$\cM_3$ is rational (over $\CC$), and the wider context into which the
proof naturally fits is the moduli of curves rather than of abelian
varieties. Katsylo's proof is one of many rationality proofs for
moduli spaces of curves: for some examples, see~\cite{Ver3}. The main
tool for many of these is classical invariant theory. Few of them are
written with much attention to fields of definition. Shepherd-Barron's
proof~\cite{S-B} that $\cM_6$ is rational over $\QQ$ is an
exception. In fact, Katsylo's proof for $\cM_3$ is among the most
complicated of these arguments.

\subsection{Prym curves}\label{subsec:prym}
A Prym curve (of genus $g$) over a scheme $S$ is a pair $(D,C)$
where $C$ is a smooth proper scheme over $S$ whose fibres are smooth
projective geometrically integral curves of genus $g$, and $D \to C$
is a finite \'etale morphism of degree $2$ whose fibres over $S$ are
geometrically integral. We denote by $\ucR_g$ the moduli stack of Prym
curves of genus $g$  and by $\cR_g$ its coarse moduli space.

An explicit construction of this stack over $\ZZ[1/2]$ can be found in
the proof of \cite[Thm.~6.5]{Bea}, which also constructs a morphism
(the Prym map) $\ucR_g \to \ucA_{g-1}$ of stacks over
$\ZZ[1/2]$. Moreover, $\ucR_g \to \ucA_{g-1}$ is dominant for $g \leq
6$, by \cite[Lem.~6.5.2]{Bea}.

\section{Rationality results}

\subsection{Stable rationality of moduli of hypersurfaces}\label{subsec:hypersurfaces}
We will prove that $\cA_3$ is stably rational by showing that the
moduli of plane quartic curves is stably rational. Our argument is
sufficiently robust that it also works for other moduli of
hypersurfaces. Our result is as follows.

\begin{theorem} \label{thm:stablyrational}
Let $d,n \geq 2$ with $(d,n) \neq (3,2)$. Let $k$ be a field and
$\cH_{d,n}$ denote the Hilbert scheme of hypersurfaces of degree $d$
in $\PP^n$. The group $\PGL_{n+1}$ acts on $\cH_{d,n}$ in a natural
way via linear change of variables.  Assume that $\characteristic(k)$
does not divide $n+1$ and that $\gcd(d,n+1)=1$. Then
$\cH_{d,n}/\PGL_{n+1}$ is stably rational.
\end{theorem}
\begin{proof}
The case $d = 2$ is classical so we assume $d > 2$.  Let
$V_{d,n}=\H^0(\PP^n, \cO_{\PP^n}(d))$ be the vector space of forms of
degree $d$ in $(n+1)$ variables.  The group $\SL_{n+1}$ acts on
$V_{d,n}$ in a natural way via linear change of variables. The
coprimality conditions ensure that the central copy of $\mu_{n+1}
\subset \SL_{n+1}$ acts faithfully on $V_{d,n}$. Moreover the induced
action of $\SL_{n+1}$ on $\cH_{d,n}$ factors through the
$\SL_{n+1}/\mu_{n+1}=\PGL_{n+1}$-action. We obtain the commutative
diagram
\[
\xymatrix{ V_{d,n}\setminus\{0\} \ar[r] \ar[d] & \cH_{d,n} \ar[d]
  \\ \big(V_{d,n}\setminus\{0\}\big)/\SL_{n+1} \ar[r] & \cH_{d,n}/\PGL_{n+1}. }
\]
The action of $\SL_{n+1}$ on $V_{d,n}$ is generically free: indeed
$\mu_{n+1}$ acts faithfully and the action of $\PGL_{n+1}$ on
$\cH_{d,n}$ is generically free as there exist hypersurfaces of degree
$d$ with trivial linear automorphism group \cite{Poo}.  We conclude
that the generic fibre of the left-hand map is an
$\SL_{n+1}$-torsor. But such a torsor is necessary trivial as
$\H^1(k,\SL_{n+1}) = 0$ for any field $k$ \cite[Lem.~2.3]{PR}. We
conclude that $V_{d,n}$ is birational to $\SL_{n+1} \times (V_{d,n}
/\SL_{n+1})$.

However, $\SL_{n+1}$ is a split reductive group, and thus rational. To
see this (well-known) fact, note that a Borel subgroup $B$ has an
opposite $B'$ such that $B\cap B'$ is a (split) maximal torus. The
product $B\cdot R_u(B')$, the big cell in the Bruhat decomposition, is
then rational as it is the product of copies of $\AA^1$ and
$\AA^1\setminus\{0\}$: see \cite[pp.~98--99]{Mil2}. Thus $V_{d,n}
/\SL_{n+1}$ is stably rational.

Next, the top arrow in the diagram is a $\Gm$-torsor. The bottom arrow
is thus a $\Gm/\mu_{n+1} = \Gm$-torsor (here the map $\mu_{n+1} \to
\Gm$ is the one induced by the $\SL_{n+1}$-action and not necessarily
the standard embedding). By Hilbert's Theorem~90 the generic fibre is
the trivial torsor: thus we see that $\cH_{d,n}/\PGL_{n+1}$ is stably
birational to $V_{d,n}/\SL_{n+1}$, which we already proved is stably
rational.
\end{proof}

\subsection{Abelian $3$-folds}\label{subsec:A3}

Let $k$ be a field of characteristic not equal to $3$.

The rational maps $\cH_{4,2}/\PGL_3 \dashrightarrow \cM_3
\dashrightarrow \cA_3$ over $k$ are birational over $\kbar$ (this
follows from Torelli and a dimension count) and hence birational over
$k$. Therefore $\cA_3$ is stably rational over $k$ by
Theorem~\ref{thm:stablyrational}.  This is sufficient for
Theorem~\ref{thm:A3stablyrational}.

\subsection{Abelian 4-folds}\label{subsec:A4}

Let $k$ be a field of characteristic $0$. We prove the $g=4$ case of
Theorem \ref{thm:unirational}.

For the following result we closely follow~\cite{ILS}, which proves
the analogous result over $\CC$, but a few modifications are
required. The main difference is that in~\cite{ILS}, the authors work
with certain quartic surfaces, and then pass to the double covers of
$\PP^3$ ramified along the quartic surface. Some care is needed over
non-algebraically closed fields where there may be many such covers
given by quadratic twists.

\begin{theorem}\label{thm:R5A4unirational}
The coarse moduli space $\cR_5$ of \'etale double covers of a genus
$5$ curve is unirational over $k$.  Hence $\cA_4$ is unirational over
$k$.
\end{theorem}

In the rest of this section, we prove the first part of this: the
second part follows since the Pyrm map $\cR_5 \to \cA_4$ is dominant
and defined over $k$.  

Fix five points $P_1,\dots,P_5 \in \PP^3(k)$ in linear general
position. Without loss of generality $P_1 =(0:0:0:1)$. We consider the
space $\cQ'\subset \H^0(\cO_{\PP^3}(4))$ of quartic forms $F$ in four
variables such that the associated quartic surface $X:=\{x\in
\PP^4\mid F(x) = 0\}$ has exactly six ordinary double points over the
algebraic closure, at least five of which are $P_1,\dots,P_5$. This is
a quasi-affine variety defined over $k$. Moreover, if $F$ is defined
over $k$ then necessarily the sixth double point is also defined over
$k$.  The key result is as follows.

\begin{lemma} \label{lem:Q_unirational}
The space $\cQ'$ is geometrically irreducible and unirational.
\end{lemma}
\begin{proof}
We follow \cite[Prop.~2.1]{ILS}, which proves unirationality over the
algebraic closure (our $\cQ'$ is the affine cone over the $\cQ$
appearing in~\cite{ILS}). The space of quartic polynomials $F$ whose
associated quartic surface has at least five double points at
$P_1,\dots,P_5$ is a vector space which we denote by $V$. For the
sixth double point, consider the scheme
\[
B_0:= \{ (F, P_0) \in V \times \PP^3 \mid F(P_0) = \partial F/\partial
x_0 (P_0) = \dots = \partial F/\partial x_3(P_0) = 0\}.
\]
The projection onto $\PP^3$ is surjective and the fibres are vector
spaces (not just affine spaces: there is a section, the zero
section). It follows that there is a dense open $U\subset \PP^3$ over
which $B_0$ becomes isomorphic to $U \times \AA^n$ for some $n$, and
therefore $B_0$ is rational.
	
The other projection is a map $B_0 \to V$ whose image contains $\cQ'$
as a dense open subset. As $B_0$ is rational, we see that $\cQ'$ is
unirational.
\end{proof}

There is a natural map
%\begin{equation} \label{eqn:Q-R_5}
$
\varrho\colon \cQ' \To \cR_5
$
%\end{equation}
defined in the following way \cite[\S1]{ILS}. For $F\in\cQ'$, let $X$ be the
associated quartic surface and
\[
\Lambda_X:=\{y^2 = F(x)\} \subset \PP(1,1,1,1,2)
\]
the associated quartic double solid. Let $W_X$ be the blow-up of
$\Lambda_X$ at $(0:0:0:1:0)$. Then composing the blow-up map with the
natural projection to $\PP^2$ yields a morphism $f\colon W_X \to
\PP^2$. This is a conic bundle morphism whose non-smooth locus $C_X
\subset \PP^2$ is a plane sextic curve whose singular points are
exactly the images of the singular points of $X$ (see \cite[Prop.~1.5,
  1.6]{ILS}). Thus if $F$ is general, $C_X$ has five ordinary double
points, so its normalisation $\widetilde{C}_X$ has genus~$5$.

Next let $S= W_X \times_{\PP^2} \widetilde{C}_X$ and let
$\widetilde{S}$ be the normalisation of $S$. We then apply Stein
factorisation to the induced map $\widetilde{S} \to \widetilde{C}_X$
to obtain a finite morphism $\Gamma_F \to \widetilde{C}_X$, which by
\cite[Prop.~1.8]{ILS} is a geometrically connected finite \'etale
cover of degree $2$. Thus the pair $(\widetilde{C}_X, \Gamma_F)$ gives
a well-defined element of $\cR_5$. All these constructions are
natural, so we define the morphism $\varrho$ by
$\varrho(F)=(\Tilde{C}_X,\Gamma_F)$.
% \eqref{eqn:Q-R_5}.

To show that $\cR_5$ is unirational, we note from \cite[Cor.~3.3]{ILS}
that $\cQ'\to \cR_5$ is dominant over $\kbar$, hence it is dominant
over $k$, whence $\cR_5$ is unirational by
Lemma~\ref{lem:Q_unirational}. As $\cR_5 \to \cA_4$ is dominant,
$\cA_4$ is also unirational. \qed

\subsection{Abelian 5-folds}\label{subsec:A5}

Let $k$ be a field of characteristic $0$. We prove the $g=5$ case of
Theorem \ref{thm:unirational}.

We proceed as in \S\ref{subsec:A4}, this time following
\cite[\S1]{FV16}, which proves the analogous result over $\CC$. This
time the method extends to $k$ with little difficulty.

\begin{theorem}\label{thm:R6A5unirational}
The coarse moduli space $\cR_6$ of \'etale double covers of a genus
$6$ curve is unirational over $k$.  Hence $\cA_5$ is unirational over
$k$.
\end{theorem}

Pick four points $O_1,\dots,O_4 \in \PP^2(k)$ in linear general
position and let $P_i = (O_i,O_i) \in \PP^2 \times \PP^2$. Consider
the linear system $\PP^{15}$ of hypersurfaces of bidegree $(2,2)$ in
$\PP^2 \times \PP^2$ that have ordinary double points at
$P_1,\,P_2,\,P_3,\,P_4$. For such a threefold $X$, projecting onto the
first factor $\PP^2$ induces a conic bundle morphism $X \to
\PP^2$. The discriminant is a sextic curve $C_X$ whose singularities
are generically the image of the singularities of $X$
\cite[Prop.~1.2]{FV16}. Thus for general $X$ there are four nodes, so
the normalisation $\widetilde{C}_X$ has genus $6$. Moreover, applying
Stein factorisation to the singular locus as in \S\ref{subsec:A4}, we
obtain a geometrically connected \'etale double cover $\Gamma_X \to
\widetilde{C}_X$. This constructs a rational map $\PP^{15}
\dashrightarrow \cR_6$ over $k$. By \cite[Thm.~1.4]{FV16} this map is
dominant, and hence $\cR_6$ is unirational. Again $\cR_6 \to \cA_5$ is
dominant, and thus $\cA_5$ is also unirational. \qed

\smallskip
Combining the results of this section with the results from
\S\ref{subsec:algclosed} completes the proof of Theorem
\ref{thm:unirational}. \qed

\section{Weak approximation}\label{sec:WA}
\subsection{$k_v$-points of a stack}\label{subsec:v-adic}
Let $k$ be a number field and $v$ a place of $k$. For finitely
presented algebraic stacks $\ucX$ over $k_v$, Christensen
\cite[\S5]{Chr} gave a topology on $\ucX\latt{k_v}$, extending the
$v$-adic topology on the usual $k_v$-points of schemes. It has the
following properties:
\begin{enumerate}
	\item Any morphism of stacks over $k_v$ induces a continuous
          map on $k_v$-points \cite[Thm.~9.0.3]{Chr}.
	\item Any smooth morphism of stacks over $k_v$ induces an open
          map on $k_v$-points \cite[Thm.~11.0.4]{Chr}.
\end{enumerate}
The topology is unique because Christensen proves in
\cite[Thm.~7.0.7]{Chr} that any $k_v$-point of an algebraic stack is
the image of a $k_v$-point under a smooth morphism from some scheme.

The following corresponds to the well-known fact that a non-empty
$v$-adic open subset of a smooth irreducible scheme is Zariski dense,
which is an application of the implicit function theorem.

\begin{lemma} \label{lem:open_dense}
Let $\ucX$ be a smooth finitely presented irreducible algebraic stack
over $k_v$ and let $W \subseteq \ucX(k_v)$ be non-empty and open in
the $v$-adic sense above. Then $W$ is dense in $\ucX$.
\end{lemma}
\begin{proof}
Let $f\colon Z \to \ucX$ be a smooth morphism from a finitely
presented irreducible scheme such that $f(Z(k_v)) \cap W \neq
\emptyset$; this exists by \cite[Thm~7.0.7]{Chr}. Then $f^{-1}(W)$ is
a non-empty $v$-adic open set, so it is dense in $Z$ as $Z$ is a
scheme.  However both $Z \to \ucX$ and $Z(k_v) \to \ucX(k_v)$ are open
and continuous. Thus $f(Z)$ is dense in $\ucX$ and $f(f^{-1}(W))$ is
dense in $f(Z)$, and the result easily follows.
\end{proof}

We require the following version of Hensel's lemma for algebraic
stacks.

\begin{lemma}\label{lem:henselstacks}
Let $R$ be a complete noetherian local ring, or an excellent Henselian
discrete valuation ring, with maximal ideal $\fm$.  Let $\ucX$ be a
smooth algebraic stack over $R$. Then for all $n \in \NN$ the natural
map
\[
\ucX(R) \To \ucX(R/\fm^n)
\]
is surjective.
\end{lemma}
\begin{proof}
Recall that smooth means formally smooth and locally of finite
presentation.  Formal smoothness (see \cite[Tag~0DNV]{Stacks}) implies
for any $n \in \NN$ and any $1$-commutative diagram
\[
\xymatrix{
	\Spec R/\fm^{n}  \ar[r] \ar[d] & \ucX \ar[d] \\
	\Spec R/\fm^{n+1} \ar[r] & \Spec R
}
\]
there exists a diagonal arrow making the diagram $2$-commutative.  We
conclude that $\ucX\latt{R/\fm^{n+1}}\to\ucX\latt{R/\fm^{n}}$ is
surjective for all $n \in \NN$, and it follows that $\lim_i
\ucX\latt{R/\fm^{n+i}}\to\ucX\latt{R/\fm^{n}}$ is surjective. By effectivity of formal objects \cite[Lem.~98.9.5, Tag 07X3]{Stacks} we have
$\lim_i \ucX(R/\fm^{n+i}) =
\ucX(\widehat{R})$, where $\widehat{R}$ denotes the completion of
$R$. Thus if $R$ is complete we are done. Otherwise, Artin
approximation \cite[Thm~1.12]{Art}, applied to the functor given by
the isomorphism classes of objects of $\ucX$, shows that the
composition $\ucX(R) \to \ucX(\widehat{R}) \to \ucX(R/\fm^{n})$ is
surjective for any $n$, as required.
\end{proof}

%\begin{remark}
%We could not find a full version of Hensel's lemma for algebraic
%stacks in the literature.  The key question that needs checking is the
%following: if $\ucX$ is an \'etale algebraic stack over a Henselian
%ring $R$, is the map $\ucX(R) \to \ucX(R/\fm)$ surjective?  If $\ucX$
%were a scheme this would follow immediately from the definition of a
%Henselian ring, but the case of stacks is less clear as the structure
%morphism $\ucX \to \Spec R$ need no longer be representable.
%\end{remark}

\subsection{Weak approximation}\label{subsec:WA}
\begin{definition}
We say that a finitely presented algebraic stack $\ucX$ over $k$
satisfies \emph{weak approximation} if the natural map
$\ucX^{\mathrm{sm}}\latt{k} \to \prod_v\ucX^{\mathrm{sm}}\latt{k_v}$
has dense image.  (Here $\ucX^{\mathrm{sm}}$ denotes the smooth locus
of $\ucX$.)
\end{definition}

Note that, unlike the case of varieties, the map $\ucX\latt{k} \to
\prod_v\ucX\latt{k_v}$ need not be injective (its injectivity is
related to triviality of various Tate--Shafarevich sets). Nevertheless
we will sometimes abuse notation and use $\ucX\latt{k}$ also to denote
the image of this map, for instance in Lemma~\ref{lem:WA_F_P}.

\begin{lemma} \label{lem:WA_criterion}
$\ucX$ satisfies weak approximation if and only if either
  $\prod_v\ucX^{\mathrm{sm}}(k_v)$ is empty or, for any finite set
  $S$ of places of $k$ and for any non-empty open subset $W \subseteq
  \prod_{v \in S} \ucX^{\mathrm{sm}}(k_v)$, there exists an element of
  $\ucX(k)$ whose images lies in $W$.
\end{lemma}
\begin{proof}
Follows immediately from the definition.
\end{proof}

A rational map $f\colon \ucX_1 \dashrightarrow \ucX_2$ of finitely
presented algebraic stacks is called \emph{birational} if there
exist dense open substacks $\ucU_i \subset \ucX_i$ such that $f$ is
defined on $\ucU_1$, $f(\ucU_1) \subseteq \ucU_2$, and
$f|_{\ucU_1}\colon \ucU_1 \to \ucU_2$ is an isomorphism. If such a
rational map exists, we say that $\ucX_1$ and $\ucX_2$ are
\emph{birationally equivalent}.

Weak approximation for smooth stacks is a birationally invariant
property.

\begin{lemma}\label{lem:WA_birational}
Let $\ucX$ and $\ucY$ be birationally equivalent smooth irreducible
algebraic stacks over $k$. Then $\ucX$ satisfies weak approximation if
and only if $\ucY$ does.
\end{lemma}

\begin{proof}
It suffices to prove the result in the case where $\ucX \to \ucY$ is
an open immersion. Weak approximation for $\ucY$ implies weak
approximation for $\ucX$ as a special case of
Lemma~\ref{lem:WA_fibration} below. On the other hand, assume that
$\ucX$ satisfies weak approximation.  If $\prod_v \ucY(k_v) =
\emptyset$, there is nothing to prove.  Otherwise, let $W \subseteq
\prod_{v \in S} \ucY(k_v)$ as in Lemma~\ref{lem:WA_criterion}. It
suffices to show that $W \cap \ucX \neq \emptyset$. However this is
Lemma~\ref{lem:open_dense} applied to $\ucY$.
\end{proof}

In particular if $\ucX$ admits an open dense substack $U$ that is
isomorphic to a scheme, then weak approximation for $\ucX$ is
equivalent to weak approximation for the scheme $U$, and the
definition does not offer anything new. Therefore to get interesting
new problems in general one should consider stacks with non-trivial
generic stabilisers; such stacks typically admit open substacks that
are gerbes over a scheme. In such cases one may hope to prove weak
approximation using the following fibration result, which is
a stacky version of \cite[Prop.~1.1]{C-TG}.

\begin{lemma} \label{lem:WA_fibration}
Let $f\colon \ucX \to \ucY$ be a smooth morphism of smooth irreducible
algebraic stacks over $k$.  Assume that $\ucY$ satisfies weak
approximation and that the fibre over every rational point of $\ucY$
is everywhere locally soluble and satisfies weak approximation. Then
$\ucX$ satisfies weak approximation.
\end{lemma}
\begin{proof}
If $\prod_v \ucX(k_v) = \emptyset$, there is nothing to prove.
Otherwise let $W \subseteq \prod_{v \in S} \ucX(k_v)$ as in
Lemma~\ref{lem:WA_criterion}. As $f$ is smooth the image $f(W)
\subseteq \prod_{v \in S} \ucY(k_v)$ is open. So let $y \in \ucY(k)$
with image in $f(W)$. Then $f^{-1}(y) \cap W \neq \emptyset$ and
$f^{-1}(y)$ is everywhere locally soluble and satisfies weak
approximation, and thus $f^{-1}(y) \cap W$ contains a rational point,
as required.
\end{proof}

For neutral affine gerbes, weak approximation is equivalent to a
statement in Galois cohomology.

\begin{lemma} \label{lem:WA_cohomology}
Let $G$ be a finite type affine group scheme over $k$. Then $BG$
satisfies weak approximation if and only if the natural map
\[
\H^1(k,G) \To \prod_{v \in S} \H^1(k_v,G)
\]
is surjective for all finite sets of places $S$ of $k$.
\end{lemma}
\begin{proof}
Recall that for any field extension $k \subset L$ we have $BG\latt{L}
= \H^1(L,G)$, since both sets classify $G_L$ torsors over $L$. The set
$\H^1(k_v,G)$ is finite \cite[Thm.~6.14]{PR} and the induced topology
is simply the discrete topology. Therefore it suffices to note that a
subset of a product of discrete sets is dense if and only if it
surjects onto any finite collection of factors.
\end{proof}

We emphasise that for a general algebraic stack $\ucX$ over $\QQ$ it
can happen that its coarse moduli space $\cX$ satisfies weak
approximation but $\ucX$ itself does not.

\begin{example} \label{ex:GW}	
Take $\ucX = B \ZZ/8\ZZ$. Then the associated coarse moduli space is
just $\Spec \QQ$, which trivially satisfies weak approximation. But
$B\ZZ/8\ZZ$ fails weak approximation (the famous example of
Wang~\cite{Wang}): there is no $\ZZ/8\ZZ$-extension of $\QQ$ that
realises the unique unramified $\ZZ/8\ZZ$-extension of $\QQ_2$.
\end{example}

For $G=\mu_2$, however, which is the case relevant to us, there is no
such problem.

\begin{lemma} \label{lem:mu_2}
If $\mu_n \subset k$, then $B \mu_n$ satisfies weak approximation.
\end{lemma}
\begin{proof}
By Kummer theory we have $\H^1(k,\mu_n) = k^\times/k^{\times n}$, and
similarly for $k_v$.  It thus suffices to apply
Lemma~\ref{lem:WA_cohomology} and note that weak approximation for $k$
implies that $k^\times /k^{\times n} \to \prod_{v \in S} k_v^\times /
k_v^{\times n}$ is surjective.
\end{proof}

From Hensel's lemma we obtain the following application of weak
approximation.

\begin{lemma} \label{lem:WA_F_P}
Let $\ucX$ be a smooth irreducible finitely presented algebraic stack
over $k$ with $\prod_v\ucX(k_v) \neq \emptyset$ that satisfies weak
approximation, and $S$ a finite collection of non-zero prime ideals of
$k$. Let $\ucX_{\cO_k}$ be a model of $\ucX$ over $\cO_k$ that is
smooth over all elements of $S$. Then the map
\[
\ucX(k) \cap \prod_{\fp \in S} \ucX_{\cO_k}(\cO_\fp)\To \prod_{\fp \in S} \ucX_{\cO_k}(\FF_\fp)
\]
is surjective.
\end{lemma}
\begin{proof}
We first note that the map $\prod_{\fp \in
  S}\ucX_{\cO_k}(\cO_{\fp})\to \prod_{\fp \in S}
\ucX_{\cO_k}(\FF_\fp)$ is surjective for all $\fp \in S$: this follows
from Hensel's lemma for stacks, Lemma~\ref{lem:henselstacks}.
Moreover the map $\ucX_{\cO_k}(\cO_{\fp}) \to \ucX_{\cO_k}(\FF_\fp)$
is continuous as $\cO_\fp\to \FF_\fp$ is continuous
\cite[Prop.~9.0.4]{Chr}.  As $\ucX_{\cO_k}(\FF_\fp)$ is discrete, we
find that the fibre of a point is an open subset of
$\ucX(k_\fp)$. Lemma~\ref{lem:WA_criterion} thus implies that it
contains a rational point, as required.
\end{proof}

\subsection{Weak approximation for $\ucA_2$} \label{subsec:WA_A_2}
Before turning to $\ucA_3$, we briefly consider the simpler case of
$\ucA_2$, where weak approximation holds. To see this, we first note
that the natural map $\ucM_2 \to \ucA_2$ is birational. Indeed the
induced map on coarse moduli spaces is birational, and furthermore the
map on generic stabilisers is an isomorphism: this is the strong Torelli
theorem given (by Serre) in \cite[Th\'eor\`eme 3]{LS}.

Therefore by Lemma \ref{lem:WA_birational} it suffices to prove that
$\ucM_2$ satisfies weak approximation. However an element of
$\ucM_2(k_v)$ is represented by a hyperelliptic curve over $k_v$, and
one can approximate this to an arbitrary precision by a hyperelliptic
curve over $k$ by simply approximating the coefficients (similarly for
any finite collection of places).

A similar proof also shows that $\ucA_1$ satisfies weak approximation.

\subsection{Weak approximation for $\ucA_3$}\label{subsec:WA_A_3}
By Theorem~\ref{thm:A3stablyrational} we know that $\cA_3$ is stably
rational over $\QQ$, and hence satisfies weak approximation
\cite[Prop.~1.2]{C-TG}.

Let us briefly consider how one would try to use this to prove that
$\ucA_3$ satisfies weak approximation, and some of the subtleties that
arise. Let $h\colon\ucA_3 \to \cA_3$ be the coarse moduli map. Let $k$ be a number field and $S$ a finite set of places of
$k$. For $v \in S$ let $A_v \in \ucA_3(k_v)$. Then there exists a
$k$-rational point $a \in \cA_3(k)$ arbitrarily close to
$a_v:=h(A_v)$. The first issue that arises is that there is no
guarantee that the fibre $h^{-1}(a)$ contains a rational point; in
classical parlance $k$ is a field of moduli for $a$, but there is no
guarantee that $k$ is a field of definition for $a$. Assuming that we
overcome this issue and find $A \in \ucA_3(k)$ with $h(A) = a$, the
second issue is as follows: we have that $h(A)$ is close to each
$h(A_v)$, but this does not guarantee that $A$ is close to the $A_v$;
in classical parlance this means that we can only guarantee that $A$
is close to some Galois twist of the $A_v$, and not our original
$A_v$.

To overcome these issues we need to understand both the generic field
of definition as well as the generic Galois twist. This is achieved by
the following result of Shimura~\cite{Shi} reformulated in stacky
language. (See~\cite{BV} for further applications of stacks to field of
moduli questions.)

\begin{lemma} \label{lem:neutralgerbe}
Let $g \in \NN$ and $h\colon \ucA_g \to \cA_g$ the coarse
moduli map. There exists a dense open subset $V \subset \cA_g$ such
that $U:=h^{-1}(V) \to V$ is a $\mu_2$-gerbe. This gerbe is neutral if
and only if $g$ is odd.
\end{lemma}
\begin{proof}
It is well known that the generic ppav has automorphism group $\mu_2$;
this implies that the generic fibre is a $\mu_2$-gerbe and one finds
$V$ by spreading out, cf.~\cite[\S{3.2}]{Poo2}.

It remains to show that the generic gerbe is neutral if and only if
$g$ is odd.  Let $K$ be the function field of $\cA_g$. Then the
generic gerbe is neutral if and only if the generic fibre has a
$K$-point, which means exactly that the generic ppav has a model over
$K$. The main result of~\cite{Shi} says this happens if and only if
$g$ is odd, as required.
\end{proof}

\subsection*{Proof of Theorem \ref{thm:A3weakapprox}}
By Lemma~\ref{lem:neutralgerbe} we know that $\ucA_3$ is birational to
a neutral $\mu_2$-gerbe over $\cA_3$.  The latter satisfies weak
approximation, and the fibres have rational points and satisfy weak
approximation by Lemma~\ref{lem:mu_2}. Thus the result follows from
Lemmas~\ref{lem:WA_birational} and~\ref{lem:WA_fibration}.  \qed

\begin{remark}
Our proof makes essential use of the fact that $g=3$ is odd. In fact,
the proof as written does not apply in the apparently easier case of
$g=2$ from~\S\ref{subsec:WA_A_2}, where Shimura showed that the
generic hyperelliptic curve of even genus does not descend to its
field of moduli \cite[Thm.~3]{Shi}. Then one might wonder whether, on
the contrary, the simple proof for $g=2$ can be adapted to the case of
$g=3$. This does not seem possible either since a key difference is
that the map $\ucM_3\to \ucA_3$ is no longer birational, despite
inducing a birational map on the coarse moduli spaces, because the
generic stabilisers no longer agree. The following gives an arithmetic
manifestation of this phenomenon.
\end{remark}

\begin{lemma}\label{lem:nonjacobian3fold}
Let $k$ be a field and $C$ a smooth quartic curve over $k$.  Then for
each separable quadratic extension $k \subset L$ there exists a
principally polarised abelian threefold $A$ over $k$ such that $A_L
\cong J(C)_L$ but $A$ is not in the image of map $\ucM_3(k) \to
\ucA_3(k)$.
\end{lemma}
\begin{proof}
The strong Torelli theorem~\cite[Th\'eor\`eme 3]{LS} mentioned
in~\S\ref{subsec:WA_A_2} above shows that $\Aut J(C) = \Aut C \oplus
\ZZ/2\ZZ$.  As $\H^1(k,\ZZ/2\ZZ)$ parametrises separable quadratic
extensions of $k$, we see that for each such extension $k \subset L$
there is a quadratic twist $A$ of $J(C)$ by $L$.  Such a twist cannot
arise from any twist of $C$, thus $A$ is not the Jacobian of any curve
of genus $3$ over $k$.
\end{proof}

\subsection*{Proof of Theorem \ref{thm:prescribedreductions}}
Immediate from Theorem~\ref{thm:A3weakapprox} and
Lemma~\ref{lem:WA_F_P}.

\subsection*{Proof of Theorem \ref{thm:BL}}
As $g \geq 7$ the coarse moduli space has general type. Let $p$ be a
prime. To prove the result we will compare the cardinalities of the
two sets
\[
I_1=\im(\ucA_g(\FF_p) \To \cA_g(\FF_p)), \quad
I_2=\im(\ucA_g(\QQ) \cap \ucA_g(\ZZ_p) \To \cA_g(\FF_p)).
\]
The set $I_1$ is simply the collection of elements of $\cA_g(\FF_p)$
that arise from some ppav over $\FF_p$. The set $I_2$ is the
collection of elements of $\cA_g(\FF_p)$ that arise as the reduction
modulo $p$ of a ppav over $\QQ$ with good reduction at $p$ (note that
$I_2 \subseteq I_1$). We will show that $I_1$ has more elements than
$I_2$ for all sufficiently large $p$.

Firstly the Lang-Weil estimates \cite{LW} imply that $|\cA_g(\FF_p)|
\sim p^m$ as $p \to \infty$, where $m = g(g+1)/2 = \dim \cA_g$. By
Lemma~\ref{lem:neutralgerbe}, spreading out, and Lang--Weil, we have
\[
\# \{ a \in \cA_g(\FF_p) \mid h^{-1}(a) \text{ is not a $\mu_2$-gerbe}\} = O(p^{m-1}).
\]
However a $\mu_2$-gerbe over $\FF_p$ is necessarily neutral for $p
\neq 2$: indeed $\mu_2$-gerbes are classified by $\H^2(\FF_p,\mu_2)$
\cite[Thm~12.2.8]{Ols}. But Kummer theory implies that
$\H^2(\FF_p,\mu_2) = (\Br \FF_p)[2]$, which is trivial as $\Br\FF_p =
0$. We therefore deduce that $|I_1| \sim p^m$ as $p \to \infty$

For $I_2$, as $\cA_g$ has general type, the Bombieri--Lang conjecture
\cite[Conj.~9.5.11]{Poo2} predicts that $\cA_g(\QQ)$ is not Zariski
dense. Thus the Lang--Weil estimates imply that $|I_2| =
O(p^{m-1})$. Altogether we deduce that for all sufficiently large
primes $p$ we have $|I_1| > |I_2|$, which concludes the proof.  \qed


\begin{thebibliography}{DSMS}

\bibitem[Art]{Art}
  M. Artin, Algebraic approximation of structures over complete local
  rings. \textit{Inst. Hautes \'Etudes Sci. Publ. Math.} \textbf{36} (1969), 23--58.

%\bibitem[AT]{AT}
%  E. Artin, Class field theory.  AMS Chelsea Publishing, Providence,
%  RI, 2009.
%
\bibitem[Bea]{Bea}
  A. Beauville, Prym varieties and the Schottky
  problem. \textit{Invent. Math.} \textbf{41} (1977), 149--196.

\bibitem[B\"oh]{Bo}
  C. B\"ohning, The rationality of the moduli space of curves of genus
  3 after P. Katsylo, \textit{Cohomological and geometric approaches
    to rationality problems}, 17--53, Progr. Math. \textbf{282}, Birkh\"auser
  Boston, Boston, MA, 2010.% {\tt arXiv:0804.1509}

\bibitem[BV]{BV}
  G. Bresciani \& A. Vistoli, Fields of moduli and the arithmetic of
  tame quotient singularities, \texttt{arXiv:2210.04789}.

\bibitem[Cle]{Cle}
  H. Clemens, Double solids. \textit{Adv. Math.} \textbf{47} (1983), 107--230.
  
\bibitem[C-TG]{C-TG}
  J.-L. Colliot-Th\'el\`ene \& P. Gille, Remarques sur l'approximation
  faible sur un corps de fonctions d'une variable.  \textit{Arithmetic
    of higher-dimensional algebraic varieties\/} (Palo Alto, CA, 2002),
  121--134, Progr. Math. \textbf{226}, Birkh\"auser Boston, Boston, MA, 2004.

\bibitem[Chr]{Chr}
  A. Christensen, A Topology on Points on Stacks. \texttt{arXiv:2005.10231}.

\bibitem[DSMS]{DSMS}
  M. Dittmann, R. Salvati Manni \& N. Scheithauer, Harmonic theta series and the Kodaira dimension of $\cA_6$.  \textit{Algebra
  Number Theory} \textbf{15}, 271--285 (2021).

\bibitem[Don]{Don}
  R. Donagi, The unirationality of $\ucA_5$,
  \textit{Ann. Math.} \textbf{119} (1984), 269--307.

\bibitem[FV]{FV16}
  G. Farkas \& A. Verra, The universal abelian
  variety over $\cA_5$. \textit{Ann. Sci. \'Ec. Norm. Sup\'er.}
  \textbf{49} (2016), 521--542.

%\bibitem[HinSil]{HinSil}
%  M. Hindry \& J. H. Silverman, \textit{Diophantine geometry. An
%    introduction.}  Graduate Texts in Mathematics \textbf{201}, Springer-Verlag, New York, 2000.
%
\bibitem[HS]{HS}
  K. Hulek \& G.K. Sankaran, The geometry of Siegel modular
  varieties. \textit{Higher Dimensional Birational Geometry}, Advanced
  Studies in Pure Mathematics \textbf{35} (2002), 89--156.

\bibitem[Igu]{Ig}
  J.-I. Igusa, Arithmetic variety of moduli for genus $2$. \textit{Ann. Math.} \textbf{72} (1960), 612--649.

\bibitem[ILS]{ILS}
  E. Izadi, M. Lo Giudice \& G.K. Sankaran, The
  moduli space of \'etale double covers of genus $5$ curves is
  unirational. \textit{Pacific J. Math.} \textbf{239} (2009), 39--52.

\bibitem[Kat]{Kat}
  P. Katsylo, Rationality of the moduli variety of curves of genus
  $3$. \textit{Comment. Math. Helvetici} \textbf{71}, 507--524 (1996).

\bibitem[LW]{LW}
  S. Lang \& A. Weil, Number of points of varieties in
  finite fields. \textit{Amer. J. Math.} {\bf 76} (1954), 819--827.

\bibitem[LM-B]{LM-B}
  G. Laumon \& L. Moret-Bailly, \textit{Champs alg\'ebriques}.
  Ergeb. Math. Grenzgeb. \textbf{39}, Springer-Verlag, Berlin, 2000.
  
\bibitem[LS]{LS}
  K. Lauter, Geometric methods for improving the upper bounds on the
  number of rational points on algebraic curves over finite fields.
  With an appendix by J.-P. Serre.  \textit{J. Algebraic Geom.}
  \textbf{10} (2001), 19--36.

\bibitem[MZ]{MZ}
  D. Masser \& U. Zannier, Abelian varieties isogenous to no Jacobian,
  \textit{Ann. Math.} \textbf{191} (2020), 635--674.

\bibitem[Mil]{Mil2}
  J.S. Milne, \textit{Reductive Groups.}
  \texttt{https://www.jmilne.org/math/CourseNotes/RG.pdf}, 2018.

\bibitem[MM]{MM}
  S. Mori \& S. Mukai, \textit{The uniruledness of the
  moduli space of curves of genus 11}. Algebraic geometry
  (Tokyo/Kyoto, 1982), 334--353, Lecture Notes in Math. \textbf{1016},
  Springer, Berlin, 1983.

%\bibitem[NSW]{NSW}
%  J. Neukirch, A. Schmidt \& K. Wingberg,
%  \textit{Cohomology of Number Fields}.  Grundlehren
%  der Mathematischen Wissenschaften \textbf{323}, Springer-Verlag, 2008.

\bibitem[Ols]{Ols}
  M. Olsson, \textit{Algebraic spaces and
  stacks}. American Mathematical Society Colloquium Publications,
  \textbf{62}. American Mathematical Society, Providence, RI, 2016.

\bibitem[Oor]{Oor}
  F. Oort, Finite group schemes, local moduli for
  abelian varieties, and lifting problems. \textit{Compositio Math.}
  \textbf{23} (1971), 265--296.

\bibitem[PR]{PR}
  V. Platonov \& A. Rapinchuk. \textit{Algebraic groups
  and number theory}. Pure and Applied Mathematics, \textbf{139.}
  Academic Press, Boston, MA, 1994.% {\rm xii}+614 pp.

\bibitem[Poo1]{Poo}
  B. Poonen, Varieties without extra
  automorphisms. III. Hypersurfaces. \textit{Finite Fields Appl.}
  \textbf{11} (2005), 230--268.

\bibitem[Poo2]{Poo2}
  B. Poonen, \textit{Rational Points on Varieties}. Graduate Studies
  in Mathematics \textbf{186}. American Mathematical Society,
  Providence, RI, 2017. 
  
\bibitem[Shi]{Shi}
  G. Shimura, On the field of rationality for an
  abelian variety. \textit{Nagoya Math. J.} \textbf{45} (1972), 167--178.

\bibitem[S-B]{S-B}
  N.I. Shepherd-Barron, Invariant theory for $S_5$ and
  the rationality of $\cM_6$.  \textit{Compos. Math.} \textbf{70}, 13--25 (1989).

\bibitem[SP]{Stacks}
  {The {Stacks Project Authors}}, \textit{Stacks Project},
  \texttt{https://stacks.math.columbia.edu}, 2023.

\bibitem[Tai]{Tai}
  Y. Tai, On the Kodaira dimension of the moduli space
  of Abelian varieties.  \textit{Invent. Math.} \textbf{68}, 425--439 (1982).

\bibitem[Ver1]{Ver1}
  A. Verra, \emph{On the universal principally
  polarized abelian variety of dimension 4}, Curves and abelian
  varieties, 253--274, Contemp. Math. \textbf{465}, Amer. Math. Soc.,
  Providence, RI, 2008.

\bibitem[Ver2]{Ver2}
  A. Verra, A short proof of the unirationality
  of $\ucA_5$. \textit{Nederl. Akad. Wetensch. Indag. Math.} \textbf{46} (1984), 
  339--355.

\bibitem[Ver3]{Ver3}
  A. Verra, \emph{Rational parametrizations of moduli
  spaces of curves}, G. Farkas \& I. Morrison (eds.), Handbook of
  moduli. Volume III. Somerville, MA: International Press; Beijing:
  Higher Education Press. Advanced Lectures in Mathematics (ALM) \textbf{26},
  431--506 (2013).

\bibitem[Wang]{Wang} S. Wang, A counter-example to Grunwald's
  theorem. \textit{Ann. Math.} (2) 49 (1948), 1008--1009.

  
\end{thebibliography}
\end{document}